\documentclass[12pt, a4paper]{article}
\usepackage{amsfonts,amssymb,amsmath,amscd,latexsym,makeidx,theorem,graphics}
\usepackage{hyperref}
\usepackage{color}

\usepackage[pdftex]{graphicx}

\newtheorem{thm}{Theorem}[section]
\newtheorem{thmx}{Theorem}[section]

\newtheorem{lem}[thm]{Lemma}
\newtheorem{prop}[thm]{Proposition}

\newenvironment{proof}{\noindent\emph{Proof.}}{\hfill$\square$\medskip}
\newcommand{\I}{\mathcal{I}}
\newcommand{\D}{\Delta}

\newcommand{\R}{\mathbb{R}}
\newcommand{\ve}{\varepsilon}

\title{Non-radial solutions to a bi-harmonic equation with negative exponent}
\author{Ali Hyder\thanks{The author is supported by the Swiss National Science Foundation, Grant No. P2BSP2-172064}\\ {\small Department of Mathematics, University of British Columbia, Vancouver BC V6T1Z2, Canada}\\ {\small \texttt{ali.hyder@math.ubc.ca} }\and Juncheng Wei\thanks{ The research is partially supported by NSERC} \\  {\small  Department of Mathematics, University of British Columbia, Vancouver BC V6T1Z2, Canada}\\  {\small \texttt{jcwei@math.ubc.ca} } }

\begin{document}

\date{}
\maketitle

\abstract  We prove the existence of non-radial  entire solution to $$\Delta^2 u+u^{-q}=0\quad\text{in }\mathbb{R}^3,\quad u>0,$$ for $q>1$. This answers an open question raised by P. J. McKenna and W. Reichel (E. J. D. E.  \textbf{37} (2003) 1-13).

\section{Introduction}
We consider the following bi-harmonic equation with negative exponent
\begin{align}\label{eq-4th}
\D^2 u+u^{-q}=0\quad\text{in }\R^3,\quad u>0,
\end{align}
where $q>0$.

For $q=7$, problem \eqref{eq-4th} can be seen as a fourth  order analog of the Yamabe equation (see \cite{Branson, Choi-Xu, Yang-Zhu}), namely \begin{align}\label{eq-nth}\D^2 u=\frac{n-4}{2}u^\frac{n+4}{n-4}\quad\text{in }\R^n,\quad u>0.\end{align}  


In the recent past,  radial solutions to equation \eqref{eq-4th} have been studied by many authors, especially the existence and asymptotic behavior: 

\begin{thmx}(\cite{Choi-Xu, Ngo2,  Guerra, Lai, MR, Xu})\label{thmA}
\begin{itemize}
\item[i)] There is no entire solution to \eqref{eq-4th} for $0< q\leq1$.
\item[ii)] If $u$ has exact linear growth at infinity, that is,
$$ \lim_{|x| \to +\infty} \frac{ u(x)}{|x|} = C>0,$$
then $q>3$. Moreover, for $q=7$, $u$ is given by $u(x)=\sqrt{\sqrt{1/15}+|x|^2}$, and is unique up to dilation and translations.
\item[iii)] For $q>3$ there exists radial solution with exact linear growth.
\item[iv)] For $q>1$ there exists radial solution with exact quadratic growth, that is, $$ \lim_{|x| \to +\infty} \frac{ u(x)}{|x|^2} = C>0.$$
\item[v)] For $1<q<3$ there exists a radial solution $u$ such that $r^{-\frac{4}{q+1}}u(r)\to C(q)>0$ as $r\to\infty$ (the constant $C(q)$ is explicitly known).
\item[vi)] For $q=3$ there exists a radial solution $u$ such that $r^{-1}(\log r)^{-\frac14}u(r)\to2^\frac14$ as $r\to\infty$.
\end{itemize}
\end{thmx}

It has been shown by Choi-Xu \cite{Choi-Xu} that if $u$ is a solution to \eqref{eq-4th} with $q>4$, and $u$ has   exact linear growth at infinity then $u$ satisfies the integral equation \begin{align}\label{eq-int} u(x)=\frac{1}{8\pi}\int_{\R^3}\frac{|x-y|}{u^q(y)}dy+\gamma, \end{align} for some $\gamma\in\R$, and $\gamma=0$ if and only if $q=7$. In fact, every positive solution $u$ to $$(-\D)^nu+u^{-(4n-1)}=0\quad\text{ in } \R^{2n-1},\quad n\geq 2,$$ with exact linear growth at infinity  satisfies $$u(x)=c_n\int_{\R^{2n-1}}\frac{|x-y|}{u^{4n-1}(y)}dy,$$ where $c_n$ is a dimensional constant, see  \cite{F-X},  \cite{Ngo}.  For the classification of solutions to the above integral equation we refer the reader to  \cite{Li}, \cite{Xu}. 

In \cite{MR} McKenna-Reichel proved the  existence of non-radial solution to \begin{align}\label{eq-n} \D^2w+w^{-q}=0\quad\text{in }\R^n,\quad w>0\end{align} for $n\geq 4$. This was a simple consequence of their existence results to \eqref{eq-n} in lower dimension. More precisely, if $u$ is a  radial solution to \eqref{eq-n}  with $n\geq 3$ then $w(x):=u(x')$ is  a non-radial solution to $\D^2 w+w^{-q}=0$ in $\R^{n+1}$,  where $x=(x',x'')\in \R^n\times \R$.  Then they asked whether in $\R^3$ non-radial positive entire solution exist. (See [Open Questions (1), \cite{MR}].) 

We answer this question affirmatively. (See Theorem \ref{thm-1} below.)   In fact we prove the following  theorems.

\begin{thm}\label{thm-classi} Let $u$ be a solution to \eqref{eq-4th} for some $q>1$. Assume that \begin{align}\label{beta}\beta:=\frac{1}{8\pi}\int_{\R^3}u^{-q}dx<+\infty.\end{align} Then, up to a rotation and  translation,  we have   \begin{align}\label{u-repre} u(x)=(\beta+o(1))|x|+\sum_{i\in\mathcal{I}_1}a_ix_i^2+\sum_{i\in\mathcal{I}_2}b_ix_i+c,\quad o(1)\xrightarrow{|x|\to\infty}0,\end{align} 
where $$\I_1,\I_2\subseteq \{1,2,3\},\quad \I_1\cap \I_2=\emptyset,\quad a_i>0\text{ for }i\in\I_1,\quad |b_i|<\beta\text{ for }i\in\I_2, \quad c>0.$$

\end{thm}
\begin{thm}\label{thm-1}
Let $q>1$. Then for every  $0<\kappa_1< \kappa_2$ there exists a non-radial solution $u$  to \eqref{eq-4th} such that
\begin{align}\label{kappa}
\liminf_{|x|\to\infty}\frac{u(x)}{|x|^2}=\kappa_1\quad\text{and }\limsup_{|x|\to\infty}\frac{u(x)}{|x|^2}=\kappa_2.
\end{align} 
\end{thm}

\begin{thm}\label{thm-2}
Let $q>7$. Then for every $\kappa>0$ there exists a non-radial solution $u$  to \eqref{eq-4th} such that $$\liminf_{|x|\to\infty}\frac{u(x)}{|x|}\in (0,\infty)\quad\text{and }\limsup_{|x|\to\infty}\frac{u(x)}{|x|^2}=\kappa.$$
\end{thm}


The non-radial solutions constructed in Theorem \ref{thm-1} also satisfies the following integral condition 
\begin{equation}
\int_{\R^3} u^{-q} dx <+\infty,
\end{equation}
for $q>\frac32$. Note that McKenna-Reichel's non-radial example has infinite $L^1$ bound: $ \int_{\R^{n+1}} w^{-q}dx=+\infty$.

The existence of infinitely many entire non-radial solutions with different growth rates for
the  conformally invariant equation $ \Delta^2 u+ u^{-7}=0 $ in $\R^3$
is in striking contrast to
other conformally invariant equations $ -\Delta u= u^{\frac{n+2}{n-2}} $ in $\R^n$,  $n\geq 3$ and  $ (-\Delta )^m u = u^{ \frac{ n+2m}{n-2m}} $ in $ \R^n, \ n>2m$. In both cases all solutions are radially symmetric with respect to some point in $\R^n$, see \cite{CGS}, \cite{CL}, \cite{Lin} and \cite{W-X}.

Our motivation in the proof of Theorems \ref{thm-1}-\ref{thm-2} come from a similar phenomena exhibited in the following equation
\begin{equation}
\label{liouville}
 (-\Delta)^{\frac{n}{2}} u= e^{n u} \ \ \mbox{in} \ \R^n, \ \ \int_{\R^n} e^{nu} dx<+\infty.
 \end{equation}
It has been proved that for $ n\geq 4$ problem (\ref{liouville}) admits non-radial entire solutions with polynomial growth at infinity, see \cite{CC}, \cite{H}, \cite{H-M},   \cite{Mar0}, \cite{M},  \cite{WY} and the references therein. It is surprising to see that conformally invariant equations with negative powers share similar phenomena. 

In the remaining part of the paper we prove Theorems \ref{thm-classi}-\ref{thm-2} respectively. We also give a new proof of $iii)$-$iv)$ of Theorem \ref{thmA},  see sub-section \ref{new-proof}.

\section{Proof of the theorems}

We begin by proving Theorem \ref{thm-classi}.

\medskip

\noindent\emph{\textbf{Proof of Theorem \ref{thm-classi}}}
Let $u$ be a solution to \eqref{eq-4th}-\eqref{beta}. We set  \begin{align} \label{def-v}v(x):=\frac{1}{8\pi}\int_{\R^3}\frac{|x-y|-|y|}{u^q(y)}dy,\quad w:=u-v. \end{align}
Fixing $\ve>0$ and $R=R(\ve)>0$ so that $$\int_{B_R^c}\frac{dx}{u^q(x)}<8\pi\ve,$$ one gets $$v(x)\geq  \frac{1}{8\pi}\int_{B_R}\frac{|x|-2|y|}{u^q(y)}dy- \frac{1}{8\pi}\int_{B_R^c}\frac{|x|}{u^q(y)}dy\geq (\beta-2\ve)|x|-C(R).$$ Using that $||x-y|-|y||\leq |x|$, form \eqref{def-v}, we obtain  $$|v(x)|\leq \beta|x|\quad\text{in }\R^3.$$ Combining these estimates we deduce that  $$\lim_{|x|\to\infty}\frac{v(x)}{|x|}=\beta.$$ 
It follows that  $w$ satisfies   $$\D^2w=0\quad\text{ in }\R^3,\quad w(x)\geq -\beta|x|,$$ and hence,  $w$ is a polynomial of degree at most $2$, see for instance \cite[Theorem 5]{Mar0}. Indeed, up to a rotation and translation, we can write  $$w(x)=\sum_{i\in\mathcal{I}_1}a_ix_i^2+\sum_{i\in\mathcal{I}_2}b_ix_i+c_0,$$ where $\I_1,\I_2$ are two disjoint (possibly empty)  subsets of $\{1,2,3\}$, $a_i\neq 0$ for $i\in\I_1$,  $b_i\neq 0$ for $i\in\I_2$ and $c_0\in\R$.  Therefore, up to a rotation and translation,  we have $$ u(x)=\frac{1}{8\pi}\int_{\R^3}\frac{|x-y|-|y|}{u^q(y)}dy+\sum_{i\in\mathcal{I}_1}a_ix_i^2+\sum_{i\in\mathcal{I}_2}b_ix_i+c.$$   Now $u>0$ and  $|v(x)|\leq \beta|x|$ lead to  $a_i>0$ for $i\in\I_1$,  $|b_i|\leq\beta$ for $i\in\I_2$ and $c=u(0)>0$.

In order to prove  that  $|b_i|<\beta$ we assume by contradiction that $|b_{i_0}|=\beta$ for some $i_0\in\I_2$. Up to relabelling we may assume that  $i_0=1$. Then   $$u(x)\leq C+|b_1x_1|+b_1x_1\quad\text{on }\mathcal{C}:=\{x=(x_1,\bar x)\in\R\times\R^2:|\bar x|\leq 1\},$$  a  contradiction to \eqref{beta}.

We conclude the proof. 
\hfill $\square$

\medskip

Now we move on to the existence results. We look for solutions to \eqref{eq-4th} of the form $u=v+P$ where $P$ is a polynomial of degree $2$.
Notice that  $u=v+P$ satisfies \eqref{eq-4th} if and only if $v$ satisfies
\begin{align}\label{v+P}\D^2 v=-(v+P)^{-q},\quad v+P>0.
\end{align} In particular, if $P\geq 0$, and   $v$ satisfies the integral equation
\begin{align}\label{eq-v}
v(x)=\frac{1}{8\pi}\int_{\R^3}|x-y|\frac{1}{(P(y)+v(y))^q}dy,
\end{align}
then  $v$ satisfies \eqref{v+P}.  Thus, we only need to find solutions to \eqref{eq-v} (or a variant of it), and we shall do that by a fixed point argument.
Let us first define the spaces on which we shall work:
$$X:=\left\{v\in C^0(\R^3): \|v\|_X<\infty\right\},\quad \|v\|_X:=\sup_{x\in\R^3}\frac{|v(x)|}{1+|x|},$$
$$X_{ev}:=\left\{v\in X: v(x)=v(-x)\,\forall x\in\R^3\right\},\quad \|v\|_{X_{ev}}:=\|v\|_X,$$   $\rule{0cm}{.2cm}$
$$X_{rad}:=\left\{v\in X: v\text{ is radially symmetric}\right\},\quad \|v\|_{X_{rad}}:=\|v\|_X.$$


The following proposition is crucial in proving Theorem \ref{thm-1}.
\begin{prop}\label{prop1}
Let $P$ be a positive function on $\R^3$ such that $P(-x)=P(x)$ and for some $q>0$ $$\int_{\R^3}\frac{|x|}{(P(x))^q}dx<\infty.$$ Then there exists a  function $v\in X_{ev}$ satisfying $\min_{\R^3}v=v(0)=0$,
\begin{align}\label{v-soln2}
 v(x):=\frac{1}{8\pi}\int_{\R^3}\frac{|x-y|-|y|}{(P(y)+v(y))^q}dy,
 \end{align}
and
 $$\lim_{|x|\to\infty}\frac{ v(x)}{|x|}=\alpha_{P,v}:=\frac{1}{8\pi}\int_{\R^3}\frac{dy}{(P(y)+v(y))^q}.$$
 Moreover, if $P$ is radially symmetric then there exists a solution   to \eqref{v-soln2} in $X_{rad}$.
\end{prop}
\begin{proof}
Let us define an operator $T:X_{ev}\to X_{ev}$, $v\mapsto\bar v$, (In case $P$ is radial we restrict the operator  $T$ on $X_{rad}$. Notice that $T(X_{rad})\subset X_{rad}$.) where
\begin{align}\label{defT}
\bar v(x):=\frac{1}{8\pi}\int_{\R^3}\frac{|x-y|-|y|}{(P(y)+|v(y)|)^q}dy.
\end{align}
We proceed by steps.

\medskip

\noindent\textbf{Step 1} $T$ is compact.

Using  that $||x-y|-|y||\leq |x|$ we bound
\begin{align}\label{est-v}
\|\bar v\|_X\leq \frac{1}{8\pi}\int_{\R^3}\frac{1}{(P(y))^q}dy\leq C\quad \text{for every }v\in X.
\end{align}
Differentiating under the integral sign one gets
$$|\nabla \bar v(x)|\leq \frac{1}{8\pi}\int_{\R^3}\frac{1}{(P(y))^q}dy\leq C\quad\text{for every }v\in X,\,x\in\R^3.$$
We let $(v_k)$ be a sequence in $X_{ev}$. Then $\bar v_k:=T(v_k)$ is bounded in $C^1_{loc}(\R^3)$.  
Moreover, up to a subsequence, for some $c_i\geq0$ with  $i=0,1$,  we have $$\frac{1}{8\pi}\int_{\R^3}\frac{|y|^i}{(P(y)+|v_k(y)|)^q}dy\xrightarrow{k\to\infty}c_i.$$
We rewrite  \eqref{defT} (with $v=v_k$ and $\bar v=\bar v_k$) as
$$\bar v_k(x)=\frac{1}{8\pi}\int_{\R^3}\frac{|x|-|y|}{(P(y)+|v_k(y)|)^q}dy+\frac{1}{8\pi}\int_{\R^3}\frac{|x-y|-|x|}{(P(y)+|v_k(y)|)^q}dy=:I_{1,k}(x)+I_{2,k}(x).$$
It follows that $$I_{1,k}(x)\to c_0|x|-c_1\quad\text{in }X.$$
Using that $||x-y|-|x||\leq |y|$ 
we bound $$|I_{2,k}(x)|\leq \frac{1}{8\pi}\int_{\R^3}\frac{|y|}{(P(y))^q}dy\leq C.$$ This implies that $$\lim_{R\to\infty}\sup_{k}\sup_{x\in\R^3\setminus B_R}\frac{I_{2,k}(x)}{1+|x|}=0.$$ Since $$\sup_{k}\sup_{x\in\R^3}|\nabla I_{2,k}(x)|<\infty,$$ up to a subsequence,
  $$I_{2,k}\to I \quad \text{in }X_{ev},$$ for some $I\in X_{ev}$. This  proves Step 1 as $T$ is continuous.

  \medskip

  \noindent\textbf{Step 2} $T$ has a fixed point in $X_{ev}$.

  It follows form \eqref{est-v} that there exists $M>0$ such that $T(X_{ev})\subset\mathcal{ B}_M\subset X_{ev}$. In particular, $T(\bar{ \mathcal{B}}_M)\subset \mathcal{B}_M$. Hence, by Schauder  fixed point theorem there exists a fixed point of $T$ in $\mathcal{B}_M$.

\medskip

  \noindent\textbf{Step 3}
$\lim_{|x|\to\infty}\frac{ \bar v(x)}{|x|}=\frac{1}{8\pi}\int_{\R^3}\frac{dy}{(P(y)+|v(y)|)^q}=:\alpha(P,v).$

Step 3 follows from
\begin{align*}
|\bar v(x)-\alpha(P,v)|x||\leq \frac{1}{8\pi}\int_{\R^3}\frac{||x-y|-|y|-|x||}{(P(y)+|v(y)|)^q}dy\leq \frac{1}{4\pi}\int_{\R^3}\frac{|y|}{(P(y))^q}dy\leq C.
\end{align*}

  \medskip

  \noindent\textbf{Step 4} If $v$ is a fixed point of $T$ then $v\geq0$.

  Differentiating  under the integral sign, from \eqref{defT} one can show that the hessian $D^2 \bar v$ is strictly positive definite, and hence $\bar v$ is strictly convex. Moreover,  using that $(P+|v|)$ is  an even function, one obtains $\nabla \bar v(0)=0$.  This leads to  $$\min_{x\in\R^3}\bar v(x)=\bar v(0)=0.  $$

We conclude the proposition.
\end{proof}

In the same spirit one can prove the following proposition.
 \begin{prop}\label{prop-2}
Let $P$ be a positive even function on $\R^3$ such that   for some $q>0$ $$\int_{\R^3}\frac{|x|}{(P(x))^q}dx<\infty.$$ Then there exists a positive function $v\in X_{ev}$ satisfying
\begin{align}\label{v-soln1}
 v(x):=\frac{1}{8\pi}\int_{\R^3}\frac{|x-y|}{(P(y)+v(y))^q}dy,\quad \min_{\R^3}v=v(0).
 \end{align}  
\end{prop}

\medskip\noindent\emph{Proof of Theorem \ref{thm-1}} Let  $q>1$ and $0<\kappa_1<\kappa_2$ be fixed. For every $\ve>0$ let $v_\ve\in X_{ev} $ be a  solution  of \eqref{v-soln2}, that is,
\begin{align}\label{vepsilon}
v_\ve(x)=\frac{1}{8\pi}\int_{\R^3}\frac{|x-y|-|y|}{(P_\ve(y)+v_\ve(y))^q}dy,
\end{align}
where $$P_\ve(x):=1+\kappa_1x_1^2+\kappa_2 (x_2^2+x_3^2)+\ve|x|^4,\quad x=(x_1,x_2,x_3)\in\R^3.$$  We claim that for every  multi-index $\beta\in \mathbb{N}^3$  with $|\beta|=2$  \begin{align} \label{fq}
|D^\beta v_\ve(x)|\leq C\quad\text{on }B_2 \quad \text{and }|D^\beta v_\ve(x)| \leq Cf_q(x)\quad\text{on }B_2^c, \end{align}
where
\begin{align*}
f_q(x):=\left\{\begin{array}{ll}
|x|^{-1}\quad&\text{if }q>3/2 \\
|x|^{-1}\log|x|\quad &\text{if }q=3/2\\
|x|^{2-2q}\quad &\text{if }q<3/2.
\end{array} \right.
\end{align*}
For $|\beta|=2$,  differentiating under the integral sign, from \eqref{vepsilon}, we obtain  
\begin{align*}
|D^\beta v_\ve((x))|&\leq C\int_{\R^3}\frac{1}{|x-y|}\frac{dy}{(P_\ve(y)+v_\ve(y))^q}\\
&\leq C\int_{\R^3}\frac{1}{|x-y|}\frac{dy}{(1+\kappa_1|y|^2)^q}\\
&=C\sum_{i=1}^3I_i(x),
\end{align*}
where $$I_i(x):=\int_{A_i}\frac{1}{|x-y|}\frac{dy}{(1+\kappa_1|y|^2)^q},\quad A_1:=B_{\frac{|x|}{2}},\, A_2:=B_{2|x|}\setminus A_1,\, A_3:=\R^3\setminus B_{2|x|}.$$  Since $q>1$ we have  $|D^\beta v_\ve|\leq C$ on $B_2$. For $|x|\geq 2$ we  bound
$$I_1(x)\leq \frac{2}{|x|}\int_{A_1}\frac{dy}{(1+\kappa_1|y|^2)^q}\leq Cf_q(x),$$
$$I_2(x)\leq \frac{C}{|x|^{2q}}\int_{A_2}\frac{dy}{|x-y|}\leq \frac{C}{|x|^{2q}}\int_{|y|\leq 3|x|}\frac{dy}{|y|}\leq C|x|^{2-2q},$$
$$I_3(x)\leq 2\int_{A_3}\frac{dy}{|y|(1+\kappa_1|y|^2)^q}\leq C|x|^{2-2q}.$$
This proves \eqref{fq}. Since $v_\ve(0)=|\nabla v_\ve(0)|=0$, by \eqref{fq},  we have
\begin{align}\label{est-v-epsilon}
v_\ve(x)\leq C\left\{\begin{array}{ll}
(1+|x|)\log(2+ |x|) \quad&\text{if }q>3/2 \\
(1+|x|)( \log(2+|x|))^2\quad &\text{if }q=3/2\\
(1+|x|)^{4-2q}\quad &\text{if }q<3/2.
\end{array} \right.
\end{align}
Therefore, for some  $\ve_k\downarrow0$ we must have  $v_{\ve_k}\to v$ in $C^3_{loc}(\R^3)$ for some $v$ in $\R^3$, where $v$  satisfies $$\D^2 v=-\frac{1}{(v+P_0)^q}\quad \text{in }\R^3,\quad v\geq 0\quad\text{in }\R^3,\quad P_0(x):=1+\kappa_1x_1^2+\kappa_2(x_2^2+x_3^2).$$ Hence,   $u=v+P_0$ is a solution to \eqref{eq-4th}. Moreover, as $v$ satisfies  \eqref{est-v-epsilon}, we have
$$\liminf_{|x|\to\infty}\frac{u(x)}{|x|^2}=\liminf_{|x|\to\infty}\frac{P_0(x)}{|x|^2}=\kappa_1,\quad  \limsup_{|x|\to\infty}\frac{u(x)}{|x|^2}=\limsup_{|x|\to\infty}\frac{P_0(x)}{|x|^2}=\kappa_2.$$ 
This completes the proof. 

$\hfill\square $

\medskip
\noindent\textbf{\emph{Proof of Theorem \ref{thm-2}}}
  Let  $q>7$ be fixed.  Then for every $\ve>0$  there exists a    positive solution $v_\ve$ to \eqref{v-soln1}  with $$P(x)=P_\ve(x):=1+\ve x_1^2+\kappa(x_2^2+x_3^2).$$ 
Setting $u_\ve:=v_\ve+P_\ve$ one gets \begin{align}\label{uepsilon} u_\ve(x)=\frac{1}{8\pi}\int_{\R^3}\frac{|x-y|}{u_\ve^q(y)}dy+ P_\ve(x),\quad \min_{\R^3}u_\ve=u_\ve(0).
\end{align}   Since    $c_q:=\frac12-\frac{3}{q-1}>0$ for $q>7$,  from \eqref{poho}, one obtains
\begin{align*}
0&=c_q\int_{\R^3}\frac{1}{u_\ve^{q-1}(x)}dx+\frac 12\int_{\R^3}\frac{2x\cdot\nabla P_\ve(x)-P_\ve(x)}{u^{q}(x)}dx\\
&=\frac12\int_{\R^3}\frac{3P_\ve(x)+2c_qu_\ve(x)-4}{u_\ve^q(x)}dx,
\end{align*}
which implies that $2c_qu_\ve(0)<4$, that is, $u_\ve(0)\leq C$. 
Therefore, by \eqref{uepsilon}
\begin{align}\label{uniform}
\frac{1}{8\pi}\int_{\R^3}\frac{|y|}{u_\ve^{q}(y)}dy=u_\ve(0)-1\leq C.
\end{align}
 Hence, differentiating under the integral sign, from \eqref{uepsilon} $$|\nabla (u_\ve(x)-P_\ve(x))|\leq \frac{1}{8\pi}\int_{\R^3}\frac{dy}{u_\ve^q(y)}\leq C.$$
Thus, $(u_\ve)_{0<\ve\leq 1}$ is bounded in $C^1_{loc}(\R^3)$. This yields $$u_\ve(x)\geq \frac{1}{8\pi}\int_{B_1}\frac{|x-y|}{u_\ve^q(y)}dy\geq \delta |x| \quad\text{for }|x|\geq 2,$$ for some $\delta>0$. Using this, and recalling that $q>4$, we deduce  $$\lim_{R\to\infty}\sup_{0<\ve\leq1}\int_{|y|\geq R}\frac{|y|}{u_\ve^q(y)}dy=0 .$$   Therefore, for some $\ve_k\downarrow 0$, we have  $u_{\ve_k}\to u$, where $u$ satisfies
\begin{align*}
u(x)=\frac{1}{8\pi}\int_{\R^3}\frac{|x-y|}{u^q(y)}dy+ 1+\kappa(x_2^2+x_3^2).
\end{align*}

We conclude the proof.
$\hfill\square $

\medskip

\subsection{A new proof of $iii)$-$iv)$ of Theorem \ref{thmA}}\label{new-proof}

\noindent\textbf{\emph{Proof of $iii)$}} Let $q>3$ be fixed. Then  by Proposition \ref{prop1}, for every $\ve>0$, there exists a radial function $u_\ve$ satisfying    $$u_\ve(x)=\frac{1}{8\pi}\int_{\R^3}\frac{|x-y|-|y|}{u_\ve^q(y)}dy+1+\ve|x|^2,\quad \min_{\R^3}u_\ve=u_\ve(0)=1.$$
Since $u_\ve$ is radially symmetric, one has (see Eq. (3.3) in \cite{Choi-Xu}) $$u_\ve(r)\geq \delta  (1+ r^4)^\frac{1}{q+1},$$ for some $\delta>0$.  Therefore, as $q>3$
$$\int_{\R^3}\frac{dx}{u_\ve^q(x)}\leq C\int_{\R^3}\frac{dx}{(1+|x|^4)^\frac{q}{q+1}}\leq C,$$ which gives  $$|\nabla u_\ve(x)|\leq \frac{1}{8\pi}\int_{\R^3}\frac{1}{u_\ve^q(y)}dy+2\ve|x|\leq C+2\ve|x|.$$   As $u_\ve(0)=1$, one would get $$u_\ve(x)\leq 1+C|x|+C\ve |x|^2.$$ Thus,  the family $(u_\ve)_{0<\ve\leq 1}$ is bounded in $C^1_{loc}(\R^3)$. Hence, for some $\ve_k\downarrow 0$ we have  $u_{\ve_k}\to u$ where $u$ satisfies  $$u(x)=\frac{1}{8\pi}\int_{\R^3}\frac{|x-y|-|y|}{u^q(y)}dy+1,\quad \min_{\R^3}u=u(0)=1.$$  Finally, as before, we have   $$\lim_{|x|\to\infty}\frac{u(x)}{|x|}=\frac{1}{8\pi}\int_{\R^3}\frac{dy}{u^q(y)}.$$ 
This completes the proof of $iii)$.
\hfill $\square$

\medskip

\noindent\textbf{\emph{Proof of $iv)$}}
Let $q>1$ be fixed. Then  by Proposition \ref{prop1}, for every $\ve>0$, there exists a non-negative radial function $v_\ve$ satisfying  
 $$v_\ve(x)=\frac{1}{8\pi}\int_{\R^3}\frac{|x-y|-|y|}{(1+|y|^2+\ve|y|^4+v_\ve(y))^q}dy.$$
The rest of the proof is similar to that of Theorem \ref{thm-1}.
\hfill $\square$

\medskip


 In the spirit of \cite[Lemma 4.9]{Choi-Xu} we prove the following Pohozaev type identity.
\begin{lem}[Pohozaev identity]
Let $u$ be a positive solution to
\begin{align}\label{poho-eq}
u(x)=\frac{1}{8\pi}\int_{\R^3}\frac{|x-y|}{u^q(y)}dy+P(x),
\end{align}
for some non-negative polynomial $P$ of degree at most $2$ and  $q>4$. Then \begin{align}\label{poho}
\left(\frac12-\frac{3}{q-1}\right)\int_{\R^3}\frac{1}{u^{q-1}(x)}dx+\frac 12\int_{\R^3}\frac{2x\cdot\nabla P(x)-P(x)}{u^{q}(x)}dx=0
\end{align}
\end{lem}
\begin{proof}
Differentiating under the integral sign, from \eqref{poho-eq}$$x\cdot \nabla u(x)=\frac{1}{8\pi}\int_{\R^3}\frac{x\cdot(x-y)}{|x-y|}\frac{1}{u^q(y)}dy+x\cdot \nabla P(x).$$ Multiplying the above identity by  $u^{-q}(x)$ and integrating on $B_R$
\begin{align}\label{17}
\int_{B_R}\frac{x\cdot \nabla u(x)}{u^q(x)}dx=\frac{1}{8\pi}\int_{B_R}\int_{\R^3}\frac{x\cdot(x-y)}{|x-y|}\frac{1}{u^q(x)u^q(y)}dydx+\int_{B_R}\frac{x\cdot \nabla P(x)}{u^q(x)}dx.\end{align}
 Integration by parts yields 
\begin{align*}
\int_{B_R}\frac{x\cdot \nabla u(x)}{u^q(x)}dx&=\frac{1}{1-q}\int_{B_R}x\cdot\nabla (u^{1-q}(x))dx\\ &=-\frac{3}{1-q}\int_{B_R}u^{1-q}dx+\frac{R}{1-q}\int_{\partial B_R}u^{1-q}d\sigma.\end{align*}
Since $q>4$ and $u(x)\geq \delta |x|$ for some $\delta>0$ and $|x|$ large $$\lim_{R\to\infty}{R} \int_{\partial B_R}u^{1-q}d\sigma=0.$$
Writing $x=\frac12((x+y)+(x-y))$, and  setting $$F(x,y):=\frac{(x+y)\cdot(x-y)}{|x-y|}\frac{1}{u^q(x)u^q(y)}$$ we get
\begin{align*}
&\frac{1}{8\pi}\int_{B_R}\int_{\R^3}\frac{x\cdot(x-y)}{|x-y|}\frac{1}{u^q(x)u^q(y)}dydx \\
&=\frac12\int_{B_R} \frac{1}{u^q(x)}\left(\frac{1}{8\pi}\int_{\R^3}\frac{{|x-y|}}{u^q(y)}dy\right)dx+\frac{1}{16\pi}\int_{B_R}\int_{\R^3}F(x,y)dydx\\
&=\frac12\int_{B_R}\frac{1}{u^q(x)}(u(x)-P(x))dx+\frac{1}{16\pi}\int_{B_R}\int_{\R^3}F(x,y)dydx.
\end{align*}
Notice that  $F(x,y)=-F(y,x)$. Hence,    $$\int_{B_R}\int_{B_R}F(x,y)dydx=0,$$ and   $$\lim_{R\to\infty} \int_{B_R}\int_{\R^3}F(x,y)dydx=\lim_{R\to\infty} \int_{B_R}\int_{B_R^c}F(x,y)dydx=0,$$ where the last equality follows from $|x|u^{-q}(x)\in L^1(\R^3)$.   Combining these estimates and taking $R\to\infty$ in \eqref{17} one gets \eqref{poho}.
\end{proof}

\end{document}